\documentclass[11pt,amssymb]{amsart}
\usepackage{amssymb}
\usepackage{amsmath}
\usepackage[mathscr]{eucal}

\newcommand{\Z}{{\mathbb Z}}

\newcommand{\C}{{\mathbb C}}
\newcommand{\Q}{{\mathbb Q}}

\newcommand{\R}{{\mathbb R}}

\newcommand{\Br}{\mathrm{Br}}

\newcommand{\Ga}{\mathrm{Gal}}
\newtheorem{thm}{Theorem}[section]
\newtheorem{lemma}[thm]{Lemma}
\newtheorem{prop}[thm]{Proposition}
\newtheorem{cor}[thm]{Corollary}

\newcommand{\Hom}{\mathrm{Hom}}

\newcommand{\Fpm}{\mathrm{(F}_m'\mathrm{)}}
\newcommand{\tF}{\mathrm{(F)}}
\newcommand{\he}{H_{\mathrm{\acute{e}t}}}

\newcommand{\tT}{\mathbb{T}}
\usepackage[all,cmtip]{xy}
.240pk scaled 1200 .240pk
\begin{document}

\title[Fields of type $\tF$ and $\Fpm$]{A generalization of Serre's condition $\tF$ with applications to the finiteness of unramified cohomology}

\author[I.~Rapinchuk]{Igor A. Rapinchuk}

\begin{abstract}
In this paper, we introduce a condition $\mathrm{(F}_m'\mathrm{)}$ on a field $K$, for a positive integer $m$,
that generalizes Serre's condition (F) and which still implies the finiteness of the Galois cohomology of finite Galois modules annihilated by $m$ and algebraic $K$-tori that split over an extension of degree dividing $m$, as well as certain groups of \'etale and unramified cohomology.
Various examples of fields satisfying $\mathrm{(F}_m'\mathrm{)}$, including those that do not satisfy (F), are given.
\end{abstract}

\address{Department of Mathematics, Michigan State University, East Lansing, MI, 48824 USA}

\email{rapinch@math.msu.edu}

\maketitle

\section{Introduction}\label{S-1}

In \cite{SerreGC}, Serre introduced the following condition on a profinite group $G$:

\vskip2mm

$\mathrm{(F)}$ \ \ \ \parbox[t]{15.5cm}{{\it For every integer $m \geq 1$, $G$ has finitely many open subgroups of index $m$.}}

\vskip2mm

\noindent (see \cite[Ch. III, \S4.2]{SerreGC}). He then defined a perfect field $K$ to be {\it of type $\tF$} if the absolute Galois group $G_K = \Ga(\overline{K}/K)$ satisfies $\tF$ --- notice that this is equivalent to the fact that for every integer $m$, the (fixed) separable closure $\overline{K}$ contains only finitely many extensions of $K$ of degree $m$. This property provides a general framework for various finiteness results involving Galois cohomology, orbits of actions on rational points, etc. In particular, Serre showed that if
$K$ is a field of type $\tF$ and $\mathcal{G}$ is a linear algebraic group defined over $K$, then the Galois cohomology set $H^1(K,\mathcal{G})$ is finite (cf. \cite[Ch. III, \S4.3, Theorem 4]{SerreGC}).

Among the examples of fields of type $\tF$ given in {\it loc. cit.} are $\R$ and $\C$, finite fields, the field $C(\!(t)\!)$ of formal Laurent series over an algebraically closed field $C$ of characteristic 0, and $p$-adic fields (i.e. finite extensions of $\Q_p$). More generally, Serre noted that if $K$ is a perfect field such that $G_K$ is topologically finitely generated, then $K$ is of type $\tF$ (see \cite[Ch. III, \S4.1, Proposition 9]{SerreGC}). Using this, one can show, for example, that if $K$ is a field of characteristic 0 and $G_K$ is finitely generated, then the field $K(\!(t_1)\!) \cdots (\!(t_r)\!)$ of iterated Laurent series over $K$ is of type $\tF$ for any $r \geq 1$ (see Proposition \ref{P-Topfg}).
On the other hand, even when $k$ is a finite field of characteristic $p > 0$, it follows from Artin-Schreier theory that the (imperfect) field $L = k(\!(t)\!)$ has infinitely many cyclic extensions of degree $p$, hence is not of type $\tF.$ In this note, we would like to propose a generalization of condition $\tF$ which does hold in this case as well as in some other situations where $\tF$ fails, and which suffices to establish some finiteness results for Galois cohomology, and particularly for unramified cohomology.

So, let $K$ be a field and $m \geq 1$ be an integer prime to ${\rm char}~K.$
We will say that $K$ is {\it of type $\Fpm$} if
\vskip2mm

$\Fpm$ \ \ \ \parbox[t]{15.5cm}{{\it For every finite separable extension $L/K$, the quotient $L^{\times}/{L^{\times}}^m$ of the multiplicative group $L^{\times}$ by the the subgroup of $m$-th powers is finite.}}

\vskip2mm
\noindent We note that if $K$ is a field of type $\tF$, then it is also of type $\Fpm$ for all $m$ prime to ${\rm char}~K$ (see Lemma \ref{L-CondF}). However, condition $\Fpm$ appears to be somewhat more flexible in several respects. First, we do not require that the field $K$ be perfect. Also,
in contrast to what happens with $\tF$, one shows that if $K$ is of type $\Fpm$ for some $m$ prime to ${\rm char}~K$, then so is $K(\!(t_1)\!)\cdots (\!(t_r)\!)$ for any $r \geq 1$ (see Corollary \ref{C-FmLS}). Furthermore, in Example 2.9, we give a construction of fields of characteristic 0 that satisfy $\Fpm$ for some, but not all, integers $m$ (and hence, in particular, are not of type ${\rm (F)}$).

As in the case of condition $\tF$, the key motivation for introducing $\Fpm$ is to give a suitable context for finiteness properties in Galois cohomology. The main result is the following.

\begin{thm}\label{T-FieldF'}
Let $K$ be a field and $m$ be a positive integer prime to $\mathrm{char}~K.$ Assume that $K$ is of type $\Fpm.$ Then for any finite $G_K$-module $A$ such that $mA = 0$, the groups $H^i (K, A)$ are finite for all $i \geq 0$.
\end{thm}
The proof of this statement, which will be given in \S\ref{S-2a}, relies on the Norm Residue Isomorphism Theorem (formerly the Bloch-Kato Conjecture).

\vskip2mm

\noindent {\bf Remark 1.2.} We note that in \cite[Ch. III, \S 4]{SerreGC}, Serre only proves that condition (F) implies the finiteness of $H^1(K,A)$, where $A$ is a finite $G_K$-group (he considers both commutative and non-commutative $A$). It appears that to establish the finiteness of cohomology groups in all degrees $i \geq 1$, the use of the Bloch-Kato Conjecture cannot be avoided. We should also point out that recently, there have been attempts to approach the Bloch-Kato conjecture via an analysis of abstract properties of profinite groups in \cite{DCF}; so far, however, these efforts have shown only partial success, with no apparent bearing on Theorem \ref{T-FieldF'}. 



\vskip2mm

\addtocounter{thm}{1}

In this note, our main applications of Theorem \ref{T-FieldF'} will be several finiteness results for unramified cohomology.
First, let us recall that if $K$ is a field equipped with a discrete valuation $v$, then for any positive integer $m$ prime to the characteristic of the residue field $\kappa (v)$, any $i \geq 1$, and any $j$, there exists a {\it residue map} in Galois cohomology
$
\partial_v^i \colon H^i (K, \mu_m^{\otimes j}) \to H^{i-1} (\kappa (v), \mu_m^{\otimes (j-1)})
$
(see, e.g., \cite[Ch. II, \S7]{GaMeSe} for a construction of $\partial_v^i$ and the end of this section for all unexplained notations).
A cohomology class $x \in H^i (K, \mu_m^{\otimes j})$ is said to be {\it unramified at $v$} if $x \in \ker \partial_v^i.$ Furthermore, if $V$ is a set of discrete valuations of $K$ such that the maps $\partial_v^i$ exist for all $v \in V$, one defines the {\it degree $i$ unramified cohomology of $K$ with respect to $V$} as
$$
H^i (K, \mu_m^{\otimes j})_V = \bigcap_{v \in V} \ker \partial_v^i.
$$
Now suppose that
$X$ is a smooth algebraic variety over a field $F$ with function field $F(X).$ Then each point $x \in X$ of codimension 1 defines a discrete valuation $v_x$ on $F(X)$ that is trivial on $F$. We let
$
V_0 = \{v_x \mid x \in X^{(1)} \}
$
denote the set of all such {\it geometric places} of $F(X)$ and define
$$
H^i_{\mathrm{ur}} (F(X), \mu_m^{\otimes j}) = H^i (F(X), \mu_m^{\otimes j})_{V_0}
$$
for any positive integer $m$ invertible in $F$.\footnotemark \footnotetext{Another definition of unramified cohomology that is frequently encountered is
$$
H^i_{\mathrm{nr}} (F(X), \mu_m^{\otimes j}) = H^i (F(X), \mu_m^{\otimes j})_{V_1},
$$
where $V_1$ is the set of all discrete valuations $v$ of $F(X)$ such that $F$ is contained in the valuation ring $\mathcal{O}_v$.
Clearly, there is an inclusion $H^i_{\mathrm{nr}} (F(X), \mu_m^{\otimes j}) \subset H^i_{\mathrm{ur}} (F(X), \mu_m^{\otimes j})$, which is in fact an equality if $X$ is proper (see \cite[Theorem 4.1.1]{CT-SB}). Note that by construction, $H^i_{\mathrm{nr}} (F(X), \mu_m^{\otimes j})$ is a birational invariant of $X$.} With these notations, we have the following.

\begin{thm}\label{T-MainThm}
Let $K$ be a field and $m$ be a positive integer prime to ${\rm char}~K$. Assume that $K$ is of type $\Fpm.$
\vskip2mm

\noindent $\mathrm{(a)}$ \parbox[t]{16cm}{Suppose $C$ is a smooth, geometrically integral curve over $K$. Then the
unramified cohomology groups $H^i_{\mathrm{ur}} (K(C), \mu_m^{\otimes j})$ are finite for all $i \geq 1$ and all $j$.}

\vskip3mm

\noindent $\mathrm{(b)}$ \parbox[t]{16cm}{Let $X$ be a smooth, geometrically integral algebraic variety of dimension $\geq 2$ over $K$. Then the unramified cohomology group $H^3_{\mathrm{ur}} (K(X), \mu_m^{\otimes 2})$ is finite if and only if $CH^2 (X)/m$ is finite, where $CH^2(X)$ is the Chow group of codimension 2 cycles modulo rational equivalence.}

\end{thm}

We should point out that the finiteness statement of part (a) has been used in \cite{CRR4} and \cite{CRR5} to show that if $K$ is a field of characteristic $\neq 2$ that satisfies $\mathrm{(F}_2'\mathrm{)}$, then the number of $K(C)$-isomorphism classes of spinor groups $G = \mathrm{Spin}_n(q)$  of nondegenerate quadratic forms $q$ over $K(C)$ in $n \geq 5$ variables, as well as of groups of some other types, that have good reduction at all geometric places $v \in V_0$ is finite. In fact, this result is part of a general program of studying absolutely almost simple groups having the same maximal tori over the field of definition --- we refer the reader to \cite{CRR3} for a detailed overview of these problems and several conjectures.

Our second application deals with the unramified cohomology of algebraic tori.
Suppose $X$ is a smooth,
geometrically integral variety over a field $F$. Given a torus $\mathbb{T}$ over $X$, we define the
the degree 1 unramified cohomology $H^1_{{\rm ur}}(F(X), T)$ of the generic fiber $T = \mathbb{T}_{F(X)}$ by
$$
H^1_{{\rm ur}}(F(X),T) = {\rm Im}~(\he^1(X, \mathbb{T}) \stackrel{f}{\longrightarrow} H^1(F(X), T)),
$$
where $f$ is the natural map (see \S\ref{S-5} for the relation between this definition and the above definition of unramified cohomology using residue maps). We have the following statement.


\begin{thm}\label{T-MainThm2}
Let $X$ be a smooth, geometrically integral algebraic variety over a field $K$ and $\mathbb{T}$ be a torus over $X$ with generic fiber $T = T_{K(X)}.$ Denote by $K(X)_T$ the minimal splitting field of $T$ inside a fixed separable closure $\overline{K(X)}$ of $K(X).$ If $K$ is of type $\Fpm$ for some positive integer $m$ that is prime to ${\rm char}~K$ and divisible by $[K(X)_T:K(X)]$,
then the unramified cohomology group $H^1_{{\rm ur}}(K(X), T)$ is finite.
\end{thm}

As an application, we will see in Corollary \ref{C-RequivFinite} a consequence for the finiteness of the group of
$R$-equivalence classes of algebraic tori.


The paper is organized as follows. In \S \ref{S-2}, we discuss some key properties and examples of fields of type $\tF$ and $\Fpm$. We then establish Theorem \ref{T-FieldF'} in \S \ref{S-2a} and deduce finiteness statements for the \'etale cohomology of smooth varieties (Corollary \ref{C-FinEt}) and algebraic tori (Corollary \ref{C-FinTori}) over fields of type $\Fpm.$
In \S \ref{S-4},
we briefly review several important points of Bloch-Ogus theory and then establish Theorem \ref{T-MainThm}. Finally, Theorem \ref{T-MainThm2} is proved in \S\ref{S-5}.



\vskip5mm

\noindent {\bf Notations and conventions.} Let $X$ be a scheme. For any positive integer
$n$ invertible on $X$, we let $\mu_n = \mu_{n, X}$ be the \'etale sheaf of $n$th roots of unity on $X$.
We follow the usual notations for the Tate twists of $\mu_n$. Namely, for $i \geq 0$, we set
$
\Z / n \Z (i) = \mu_n^{\otimes i}
$
(where $\mu_n^{\otimes i}$ is the sheaf associated to the $i$-fold tensor product of $\mu_n$), with the convention that
$
\mu_n^{\otimes 0} = \Z/n \Z.
$
If $i < 0$, we let
$$
\Z / n \Z (i) = \Hom (\mu_n^{\otimes (- i)}, \Z / n \Z).
$$
In the case that $X = \mathrm{Spec}~F$ for a field $F$, we identify $\mu_n$ with the group of $n$th roots of unity in a fixed separable closure $\overline{F}$ of $F$. We will also tacitly identify the \'etale cohomology of ${\rm Spec}~F$ with the Galois cohomology of $F$.

\section{Fields of type $\tF$ and $\Fpm$}\label{S-2}

In this section, we will discuss some key properties and examples of fields of types $\tF$ and $\Fpm.$

\subsection{Fields of type $\tF$} Let $K$ be a perfect field. As we already mentioned in \S\ref{S-1}, if the absolute Galois group $G_K = \Ga(\overline{K}/K)$ is topologically finitely generated, then $K$ is automatically of type $\tF.$ Well-known examples of fields for which $G_K$ is finitely generated include $\C$, $\R$, finite fields, and $p$-adic fields (in this case, it is known that if $[K : \Q_p] = d$, then $G_K$ can be generated by $d + 2$ elements --- see \cite[Theorem 3.1]{J1}). Some further examples of fields with finitely generated absolute Galois group
can be constructed using the next observation.

\begin{prop}\label{P-Topfg}
Let $K$ be a field of characteristic 0 such that $G_K = \Ga (\overline{K}/K)$ is topologically finitely generated. Then for the field $L = K(\!(t)\!)$ of formal Laurent series over $K$, $G_L$ is also topologically finitely generated.
\end{prop}
\begin{proof}
First, we note the following:
\vskip2mm

\noindent $(*)$ \parbox[t]{14.5cm}{\it Suppose $\mathscr{K}$ is a discretely valued henselian field with residue field $\mathbf{k}$ and group of units $\mathscr{U} \subset \mathscr{K}.$ Then for any $m \geq1$ relatively prime to $\mathrm{char}\: \mathbf{k}$, the reduction map induces an isomorphism $\mathscr{U}/\mathscr{U}^m \simeq \mathbf{k}^{\times}/{\mathbf{k}^{\times}}^m$.}

\vskip2mm

\noindent (Indeed, the henselian property implies that the polynomial $X^m - u$ with $u \in \mathscr{U}$ has a root in $\mathscr{K}$ (or $\mathscr{U}$) if and only if its reduction $X^m - \bar{u}$ has a root in $\mathbf{k}$, which leads to the required isomorphism.)

Now let $F = L^{\text{ur}}$ be the maximal unramified extension of $L$. Then $F$ is a discretely valued henselian field having residue field $\overline{K}$ (in fact, in this case, $F$ is simply the compositum of $L$ and $\overline{K}$ inside $\overline{L}$). Since $\Ga(F/L)$ is naturally identified with $G_K$, hence is finitely generated, it is enough to show that $H = \Ga (\overline{L}/F ) \simeq \widehat{\Z}$. This will follow if we verify that for each $n \geq 1$, $F(\sqrt[n]{t})$ is the \emph{unique} degree $n$ extension of $F$. But, since ${\rm char}~K = 0$, a degree $n$ extension $F'/F$ is totally tamely ramified, and consequently is of the form $F(\sqrt[n]{\pi})$ for {\it some} uniformizer $\pi \in F$ (see, e.g., \cite[Ch. II, Proposition 3.5]{FV}). Since the residue field $\overline{K}$ of $F$ is algebraically closed, it follows from $(*)$ that every unit in $F$ is an $n$-th power, hence $F(\sqrt[n]{\pi}) = F(\sqrt[n]{t})$, as required.

\end{proof}

Iterating the statement of the proposition, we obtain, in particular, the following.
\begin{cor}\label{C-TopFg}
Suppose $K$ is $\C$, $\R$, or a $p$-adic field, and let $L = K(\!(t_1)\!) \cdots (\!(t_r)\!)$ for some $r \geq 0$. Then $G_L$ is topologically finitely generated, and therefore $L$ is of type $\tF.$
\end{cor}

\vskip3mm

\noindent {\bf Remark 2.3.} Let us point out that while the examples of fields of type $\tF$ given in \cite[Ch. III, \S4.2]{SerreGC} have (virtual) cohomological dimension $\leq 2$, the preceding proposition and corollary allow us to construct such fields of arbitrary (finite) cohomological dimension.

\addtocounter{thm}{1}

\subsection{Fields of type $\Fpm$} We would first like to elaborate on the relationship between conditions $\tF$ and $\Fpm$ that was mentioned in \S\ref{S-1}.


\begin{lemma}\label{L-CondF}
Suppose $K$ is a (perfect) field of type $\tF.$ Then $K$ satisfies $\Fpm$ for all $m$ prime to ${\rm char}~K.$
\end{lemma}
\begin{proof}
It is clear from the definitions that if $K$ is of type $\tF$, then so is any finite separable extension $L/K$. Thus, it suffices to show that for any $m$ prime to $\text{char}~K$, the group
$$
H^1 (K, \Z / m \Z (1)) = K^{\times}/ {K^{\times}}^m
$$
is finite. Let $L = K(\zeta_m)$, where $\zeta_m$ is a primitive $m$-th root of unity. Then
$$
H^1 (L, \Z / m \Z (1)) = \Hom_{\text{cont}} (G_L, \Z / m \Z),
$$
which is finite since $L$ is of type $\tF$. The finiteness of $H^1 (K, \Z / m \Z (1))$ then follows from the inflation-restriction sequence
$$
0 \to H^1 (\Ga (L/K), \Z / m \Z (1)) \to H^1 (K, \Z / m \Z (1)) \to H^1 (L, \Z / m \Z (1)).
$$
\end{proof}

\vskip2mm

\noindent {\bf Remark 2.5.} It follows immediately from the definition that if $K$ is of type $\Fpm$, then $K$ is automatically of type $\mathrm{(F'_{\ell})}$ for every $\ell \vert m.$

\addtocounter{thm}{1}

\vskip2mm

As we mentioned previously, condition $\Fpm$ provides greater flexibility than $\tF$ in certain common situations. For instance, whereas $K(\!(t)\!)$, for $K$ a finite field of characteristic $p$, fails to be of type $\tF$, Corollary \ref{C-FmLS} below implies that it \emph{does} satisfy $\Fpm$ for all $m$ relatively prime to $p$. Quite generally, we have the following statement.

\begin{prop}\label{P-FmLS}
Let $\mathscr{K}$ be a complete discretely valued field with residue field $\mathbf{k}$. If $m$ is prime to $\mathrm{char}\:
\mathbf{k}$ and $\mathbf{k}$ is of type $\Fpm$, then $\mathscr{K}$ is also of type $\Fpm.$
\end{prop}

In particular, for the field of Laurent series, we have

\begin{cor}\label{C-FmLS}
Let $K$ be a field and $m \geq 1$ be an integer prime to ${\rm char}~K.$ If $K$ is of type $\Fpm$, then so is $L = K(\!(t)\!).$
\end{cor}

Thus, if, for example, $K$ is $\C$, $\R$, a $p$-adic field, or a finite field of characteristic $p$ prime to $m$, then $L = K(\!(t_1)\!)\cdots (\!(t_r)\!)$ is of type $\Fpm$ for any $r \geq 1.$

Now, before proceeding to the proof of Proposition \ref{P-FmLS}, let us observe that while we restricted ourselves to finite {\it separable} extensions in the definition of $\Fpm$, the property described there actually holds for {\it all} finite extensions.

\begin{lemma}\label{L-Insep}
If a field $K$ is of type $\Fpm$ for $m$ prime to $\mathrm{char}\: K$, then the quotient $L^{\times}/{L^{\times}}^m$ is finite for \emph{any} (not necessarily separable) finite extension $L/K$.
\end{lemma}
\begin{proof}
We can assume that $p = \mathrm{char}\: K > 0$. Let $L/K$ be an arbitrary finite extension, and $F$ be its maximal separable subextension. Then there exists $\alpha \geqslant 0$ such that $(L)^{p^{\alpha}} \subset F$. On the other hand,  since $K$ is of type $\Fpm$, the quotient $F^{\times}/{F^{\times}}^m$ is finite. So, it suffices to show that the map
\begin{equation}\label{E:A-Inj}
L^{\times}/{L^{\times}}^m \longrightarrow F^{\times}/{F^{\times}}^m, \ \ x{L^{\times}}^m \mapsto x^{p^{\alpha}} {F^{\times}}^m,
\end{equation}
is injective. But the composite map
$$
L^{\times}/{L^{\times}}^m \longrightarrow F^{\times}/{F^{\times}}^m \longrightarrow L^{\times}/{L^{\times}}^m,
$$
where the second map is induced by the identity embedding $F^{\times} \hookrightarrow L^{\times}$, amounts to raising to the power $p^{\alpha}$ in $L^{\times}/{L^{\times}}^m$. Since the latter has exponent $m$, which is prime to $p^{\alpha}$, this map is injective, and the injectivity of (\ref{E:A-Inj}) follows.
\end{proof}

\noindent {\it Proof of Proposition \ref{P-FmLS}.}
Any finite extension $\mathscr{L}$ of $\mathscr{K}$ is clearly henselian, and its residue field $\mathbf{l}$ is a finite extension of $\mathbf{k}$.
So, since $\mathbf{k}$ is of type $\Fpm$, Lemma \ref{L-Insep} implies that the quotient $\mathbf{l}^{\times}/{\mathbf{l}^{\times}}^m$ is finite.
Applying $(*)$ from the proof of Proposition \ref{P-Topfg} to $\mathscr{L}$, we see that if $\mathscr{U}$ denotes the group of units in $\mathscr{L}$, then the quotient
$\mathscr{U}/\mathscr{U}^m$ is finite. A choice of uniformizer $\pi \in \mathscr{L}$ then yields an isomorphism
$$
\mathscr{L}^{\times} \simeq \langle \pi \rangle \times \mathscr{U},
$$
from which we conclude that the quotient $\mathscr{L}^{\times}/{\mathscr{L}^{\times}}^m$ is also finite, as required. $\Box$

\vskip2mm

We end this section with the following example of fields of characteristic 0 that satisfy $\Fpm$ for some, but not all, values of $m.$

\vskip2mm

\noindent {\bf Example 2.9.} Let $K$ be a field of characteristic 0 such that $G_K$ is an infinitely generated pro-$2$ group (e.g., we can take $K = \overline{\Q}^{\mathscr{G}_2}$, where $\mathscr{G}_2$ is a Sylow 2-subgroup of $G_{\Q}$). Clearly, $K$ is not of type ${\rm (F_2)}$, but is of type $\Fpm$ for any odd $m$ (in particular, $K$ is not of type $\tF$). By Corollary \ref{C-FmLS}, we also see that these properties are shared by $K(\!(t_1)\!)\cdots(\!(t_r)\!)$ for any $r \geq 1.$

\section{Proof of Theorem \ref{T-FieldF'}.}\label{S-2a}

In this section, we establish a finiteness result for the cohomology of finite Galois modules over fields of type $\Fpm$ and indicate some consequences for the \'etale cohomology of smooth algebraic varieties and the Galois cohomology of algebraic tori



One of the principal ingredients in the proof of Theorem \ref{T-FieldF'} is the Norm Residue Isomorphism Theorem (formerly the Bloch-Kato Conjecture), established in the work of Rost \cite{Rost}, Voevodsky (\cite{V1}, \cite{V2}), Weibel \cite{Weib}, and others. We briefly recall the set-up. For a field $F$ and integer
$n > 1$, the {\it $n$-th Milnor $K$-group} $K_n^M(F)$ of $F$ is defined as the quotient of the $n$-fold tensor product $F^{\times} \otimes_{\Z} \cdots \otimes_{\Z} F^{\times}$ by the subgroup generated by elements $a_1 \otimes \cdots \otimes a_n$ such that $a_i + a_j = 1$ for some $1 \leq i < j \leq n.$ For $a_1, \dots, a_n \in F^{\times}$, the image of $a_1 \otimes \cdots \otimes a_n$ in $K^M_n(F)$, denoted $\{ a_1, \dots, a_n \}$, is called a {\it symbol}; thus, $K^M_n (F)$ is generated by symbols.
By convention, one sets $K^M_0 (F) = \Z$ and $K^M_1 (F) = F^{\times}.$ Furthermore, for any integer $m$ prime to $\text{char}~F$, the Galois symbol yields a group homomorphism
$$
h^n_{F,m} \colon K_n^M (F) \to H^n (F, \Z / m \Z (n)),
$$
which obviously factors through $K_n^M(F)/m  := K_n^M(F)/m \cdot K_n^M(F)$ (see \cite[\S 4.6]{GiSz} for the relevant definitions). We will write
$\{a_1, \dots, a_n\}_m$ for the image of $\{a_1, \dots, a_n\}$ in $K_n^M(F)/m.$ The main result is
\begin{thm}\label{T-BlochKato}
For any field $F$ and any integer $m$ prime to $\text{char}~F$, the Galois symbol induces an isomorphism
$$
K_n^M(F)/m \stackrel{\sim}{\longrightarrow} H^n (F, \Z / m \Z (n))
$$
for all $n \geq 0$.
\end{thm}

\vskip2mm

We now turn to

\vskip2mm

\noindent {\it Proof of Theorem \ref{T-FieldF'}.} First, we note that if $K$ is of type $\Fpm$, then the groups $K_n^M(L)/m$ are finite for all finite extensions $L/K$ and all $n \geq 0.$ Indeed, for each $n$, the group $K_n^M(L)/m$ is generated by symbols $\{ a_1, \ldots , a_n \}_m$ with $a_i \in L^{\times}$, and each such symbol is annihilated by $m$. On the other hand, since the quotient $L^{\times}/{L^{\times}}^m$ is finite by Lemma \ref{L-Insep}, there are only finitely many symbols of length $n$, and the finiteness of $K_n^M(L)/m$ follows.

Applying Theorem \ref{T-BlochKato} and Remark 2.5, we see that the cohomology groups $H^n(L , \Z/\ell \Z(n))$ are finite for all $\ell \vert m$ and all $n \geq 0$. Now, to deal with an arbitrary module $A$, we pick a finite extension $L/K$ that contains the $m$-th roots of unity and such that $G_L$ acts trivially on $A$. Then for any $\ell \vert m$ and any $n \geq 0$, the Galois $L$-module  $\Z/\ell\Z(n)$ is isomorphic to the trivial module $\Z/\ell\Z$, so the above discussion implies the finiteness of the cohomology groups $H^n(L , \Z/\ell \Z)$. On the other hand, in view of the fact that $mA = 0$ and our choice of $L$, the Galois $L$-module $A$ is isomorphic to a direct sum of some $\Z/\ell\Z$'s. Thus, the cohomology groups $H^n(L , A)$ are finite for all $n \geq 0$. The finiteness of the cohomology groups $H^n(K , A)$ now follows from the Hochschild-Serre spectral sequence
$$
E_2^{i,j} = H^i(\Ga(L/K) , H^j(L , A)) \Longrightarrow H^{i+j}(K , A).
$$ \hfill $\Box$

\vskip2mm

As a first application of Theorem \ref{T-FieldF'}, let us mention the following statement for \'etale cohomology, which will be needed in our discussion of unramified cohomology.

\begin{cor}\label{C-FinEt}
Suppose $K$ is a field and $m$ is a positive integer prime to ${\rm char}~K.$ If $K$ is of type $\Fpm$, then for any smooth geometrically integral algebraic variety $X$ over $K$, the \'etale cohomology groups $\he^i(X, \Z / m \Z (j))$ are finite for all $i \geq 0$ and all $j$.
\end{cor}
\begin{proof}
Let $\bar{X} = X \times_K \overline{K}.$ It is well-known that $\he^i (\bar{X}, \Z / m \Z (j))$ are finite $m$-torsion groups for all $i \geq 0$ and all $j$ (see \cite[Expos\'e XVI, Th\'eor\`eme 5.2]{SGA4}). Consequently, the groups $H^p (K, \he^q (\bar{X}, \Z / m \Z (j))$ are finite for all $p, q \geq 0$ and all $j$ by Theorem \ref{T-FieldF'}. Our claim now follows from the
Hochschild-Serre spectral sequence
$$
E_2^{p,q} = H^p (K, H^q_{\text{\'et}} (\bar{X}, \Z / m \Z (j))) \Rightarrow H^{p+q}_{\text{\'et}} (X, \Z / m \Z (j)).
$$
\end{proof}

Next, suppose that $T$ is an algebraic torus defined over a field $K$ and
denote by $K_T$ the minimal splitting field of $T$ inside a fixed separable closure $\overline{K}$ of $K$.

\begin{cor}\label{C-FinTori}
Let $K$ be a field of type $\Fpm$. Then for an algebraic $K$-torus $T$, we have the following:
\vskip2mm

\noindent $\mathrm{(a)}$ \parbox[t]{16cm}{The $\ell$-torsion subgroups ${}_{\ell}H^i(K, T)$ are finite for all $i \geq 0$ and all $\ell \vert m.$}

\vskip2mm

\noindent $\mathrm{(b)}$ \parbox[t]{16cm}{If $[K_T:K]$ divides $m$, then $H^1(K,T)$ is finite.}

\end{cor}


\begin{proof}

\noindent (a) First, by Remark 2.5, $K$ is of type $\mathrm{(F'_{\ell})}$ for every $\ell \vert m.$ Next, for every $i \geq 0$, the map
$$
\varphi_{\ell} \colon T \to T, \ \ \ x \mapsto x^{\ell}
$$
induces a surjective homomorphism $H^i(K, T[\ell]) \twoheadrightarrow {}_{\ell}H^i(K, T)$, where $T[\ell] = \ker \varphi_{\ell}.$ Since $T[\ell]$ is a finite $\ell$-torsion module, the groups $H^i(K, T[\ell])$ are finite by
Theorem \ref{T-FieldF'}, which then yields the finiteness of ${}_{\ell}H^i(K, T).$

\vskip2mm

\noindent (b) By Hilbert's Theorem 90, the group $H^1(K , T)$ can be identified with $H^1(K_T/K , T)$, hence is annihilated by $n = [K_T : K]$. On the other hand, since $n \vert m$, the group ${}_n H^1(K , T)$ is finite according to (a), and our assertion follows.


\end{proof}

\noindent (We should point out that by contrast with part (b) of the corollary, the groups $H^i(K , T)$, for $i \geq 2$, may be infinite --- for instance, $K = \Q_p$ is of type $\tF$, hence also of type $\Fpm$ for any $m \geq 1$, but for the one-dimensional split torus $T = \mathbb{G}_m$, the group $H^2(K , T)$ is the Brauer group $\Br(K)$, which is infinite.)



We would like to mention one situation where one has the finiteness of $H^1(K , T)$ for all maximal $K$-tori of a semi-simple $K$-group $\mathcal{G}$ simultaneously.


\begin{cor}\label{C-Semisimple}
Let $\mathcal{G}$ be a semisimple algebraic group defined over a field $K$. Assume that the order $m$ of the automorphism group of the root system of $\mathcal{G}$ is prime to $\mathrm{char}\: K$ and that $K$ is of type $\Fpm$. Then for every maximal $K$-torus $T$ of $\mathcal{G}$, the group $H^1(K , T)$ is finite.
\end{cor}
Indeed, the Galois action on the root system $\Phi = \Phi(\mathcal{G} , T)$ yields an injective group homomorphism $\mathrm{Gal}(K_T/K) \hookrightarrow \mathrm{Aut}(\Phi)$. Thus, $[K_T : K]$ divides $m$, and our assertion follows from Corollary \ref{C-FinTori}(b).


In view of Serre's finiteness results for the Galois cohomology of linear algebraic groups over fields of type $\tF$, it would be interesting to determine if the assumptions of Corollary \ref{C-Semisimple} already imply the finiteness of $H^1(K , \mathcal{G})$.

\section{Proof of Theorem \ref{T-MainThm}}\label{S-4}

In this section, we will use Bloch-Ogus theory together with Theorem \ref{T-FieldF'} to establish Theorem \ref{T-MainThm}.

\subsection{A brief review of Bloch-Ogus theory} We begin by recalling, for the reader's convenience, several key points of Bloch-Ogus \cite{BO} theory, which allows one to relate unramified cohomology to \'etale cohomology.


Let $F$ be an arbitrary field and $X$ a smooth algebraic variety over $F$. In this subsection, we will fix a positive integer $m$ prime to ${\rm char}~F.$ By considering the filtration by coniveau, Bloch and Ogus established the existence of the following cohomological first quadrant spectral sequence called the {\it coniveau spectral sequence}:
\begin{equation}\label{E-ConiveauSS}
E_1^{p,q}(X/F, \Z / m \Z (b)) = \bigoplus_{x \in X^{(p)}} H^{q-p}(\kappa(x), \Z / m \Z (b - p)) \Rightarrow \he^{p+q}(X, \Z / m \Z (b)),
\end{equation}
where $X^{(p)}$ denotes the set of points of $X$ of codimension $p$ and the groups on the left are the Galois cohomology groups of the residue fields $\kappa(x)$ (the original statement of Bloch-Ogus was actually given in terms of \'etale homology, with the above version obtained via absolute purity; for a derivation of this spectral sequence that avoids the use of \'etale homology, we refer the reader to \cite{CTHK}). This spectral sequence yields a complex
\begin{equation}\label{E-BOComplex}
E_1^{\bullet, q}(X/F, \Z / m \Z (b)),
\end{equation}
and it is well-known (see, e.g., \cite[Remark 2.5.5]{JSS}) that the differentials in (\ref{E-BOComplex}) coincide up to sign with the differentials in an analogous complex constructed by Kato \cite{Kato} using residue maps in Galois cohomology.

The fundamental result of Bloch and Ogus was the calculation of the $E_2$-term of (\ref{E-ConiveauSS}). To give the statement, we will need the following notation: let
$\mathcal{H}^q (\Z / m \Z (j))$
denote the Zariski sheaf on $X$ associated to the presheaf that assigns to an open $U \subset X$ the cohomology group $H^i_{\text{\'et}}(U, \Z / m \Z (j)).$
Bloch and Ogus showed that
$$
E_2^{p,q}(X/F, \Z/ m \Z (b)) = H^p (X, \mathcal{H}^q (\Z / m \Z (b)))
$$
(see \cite[Corollary 6.3]{BO}). The resulting (first quadrant) spectral sequence
\begin{equation}\label{E-BO}
E_2^{p,q}(X/F, \Z / m \Z (b)) = H^p (X, \mathcal{H}^q (\Z / m \Z (b))) \Rightarrow H^{p+q}_{\text{\'et}} (X, \Z/ m \Z (b))
\end{equation}
is usually referred to as the {\it Bloch-Ogus spectral sequence}. For ease of reference, we summarize several key points pertaining to (\ref{E-BO}).

\begin{prop}\label{P-BlochOgusSS}
The Bloch-Ogus spectral sequence $(\ref{E-BO})$ associated to a smooth irreducible algebraic variety $X$ over a field $F$ has the following properties:
\vskip1mm

\noindent {\rm (a)} $E_2^{p,q} = 0$ for $p > \dim X$ and all $q$;

\vskip1mm

\noindent {\rm (b)} $E_2^{p,q} = 0$ for $p > q$; and

\vskip1mm

\noindent {\rm (c)} \parbox[t]{16cm}{$E_2^{0, q} = H^0 (X, \mathcal{H}^q (\Z / m \Z (b)))$ coincides with the unramified cohomology $H^q_{\mathrm{ur}} (F(X), \Z / m \Z (b))$ with respect to the geometric places of $F(X)$.}

\end{prop}

\subsection{Finiteness of unramified cohomology.} We now turn to the proof of Theorem \ref{T-MainThm}.

First suppose that
$X = C$ is a smooth, geometrically integral curve over a field $F$. Then by Proposition \ref{P-BlochOgusSS}(a), we have $E_2^{p,q} = 0$ for $p \neq 0,1$ and all $q$. So, in view of part (c), we have a surjection
\begin{equation}\label{E-EtUnram}
H^i_{\text{\'et}}(C, \Z / m \Z (j)) \twoheadrightarrow H^i_{\text{ur}} (F(C), \Z / m \Z (j))
\end{equation}
for all $i \geq 1$, all $j$, and all $m$ prime to ${\rm char}~F$. Now, if $F = K$ is a field of type $\Fpm$, then the groups $H^i_{\text{\'et}}(C, \Z / m \Z (j))$ are finite by Corollary \ref{C-FinEt}. Consequently, from (\ref{E-EtUnram}), we obtain



\begin{prop}\label{P-4.1}
Let $K$ be a field, $C$ a smooth, geometrically integral curve over $K$, and $m$ a positive integer prime to ${\rm char}~K.$ If $K$ is of type $\Fpm$, then the
unramified cohomology groups $H^i_{\mathrm{ur}} (K(C), \Z/ m \Z (j))$ are finite for all $i \geq 1$ and all $j$.
\end{prop}

We note that Proposition \ref{P-4.1} for $m=2$ has been used in \cite{CRR4} and \cite{CRR5} to prove the following finiteness theorem for spinor groups with good reduction: {\it Let F = K(C) be the function field of a smooth geometrically integral curve C over a field $K$ of characteristic $\neq 2$ that satisfies $\mathrm{(F}_2'\mathrm{)}$, and let $V$ be the set of geometric places of $F$ (corresponding to the closed points of $C$). Then the number of $F$-isomorphism classes of spinor groups $G = \mathrm{Spin}_n(q)$ of nondegenerate quadratic forms $q$ over $F$ in $n \geqslant 5$ variables that have good reduction at all $v \in V$ is finite.}

\vskip2mm

Next, for smooth algebraic varieties $X$ of dimension $\geq 2$, relatively little is known about finiteness properties of unramified cohomology in degrees $\geq 3$ in general. We should point out, however, that in the case of quadrics, several finiteness statements can be deduced from results of Kahn \cite{KahnLower} and Kahn, Rost, and Sujatha \cite{KRS}. More precisely, let $q$ be a nondegenerate quadratic form over a field $F$ of characteristic $\neq 2$ and denote by $X$ the associated projective quadratic. A natural approach for analyzing unramified cohomology is to
consider the restriction map
$$
\eta_2^i \colon H^i (F, \Z/2 \Z (1)) \to H^i_{\rm ur}(F(X), \Z/ 2 \Z(1)).
$$
If $\dim X > 2$, it is shown in \cite{KahnLower} that the cokernel of $\eta_2^3$ has order at most 2, while the main results of \cite{KRS} imply that $\vert {\rm coker}~\eta_2^4 \vert \leq 4$ for $\dim X > 4.$ Consequently, if $F = K$ is a field of type $\mathrm{(F}_2'\mathrm{)}$, we see that $H^3_{{\rm ur}}(K(X), \Z/ 2 \Z (1))$ and $H^4_{{\rm ur}}(K(X), \Z / 2 \Z (1))$ are finite in the respective cases.


Returning now to the general case, the edge map in (\ref{E-EtUnram}) is typically not surjective for varieties of dimension $\geq 2.$
However, the finiteness of unramified cohomology in degree 3 in this setting turns out to be closely related to the question of the finite generation of the Chow group $CH^2 (X)$ of codimension 2 cycles modulo rational equivalence. First, we note the following (well-known) general fact about first quadrant spectral sequences satisfying condition (b) of Proposition \ref{P-BlochOgusSS}.

\begin{lemma}\label{L-SpecSeqLemma}
Let $E_2^{p,q} \Rightarrow E^{p+q}$ be a first quadrant spectral sequence. Assume that $E_2^{p,q} = 0$ for $p > q$. Then there is an exact sequence
$$
E^3 \stackrel{e}{\longrightarrow} E_2^{0,3} \stackrel{d_2}{\longrightarrow} E_2^{2,2} \to E^4,
$$
where $e$ is the usual edge map and $d_2$ is the differential of the spectral sequence.

\end{lemma}





\vskip2mm

Applying Lemma \ref{L-SpecSeqLemma} to the Bloch-Ogus spectral sequence (with $b =2$), we thus obtain, for any smooth algebraic variety $X$ over an arbitrary field $F$ and any positive integer $m$ prime to ${\rm char}~F$, an exact sequence
\begin{equation}\label{E-BOChow1}
\he^3 (X, \Z / m \Z (2)) \to H^0 (X, \mathcal{H}^3(\Z / m \Z (2))) \to H^2 (X, \mathcal{H}^2 (\Z / m \Z (2))) \to \he^4 (X, \Z m \Z (2)).
\end{equation}
On the other hand, a well-known consequence of Quillen's proof of Gersten's conjecture in algebraic $K$-theory is the existence of an isomorphism
$$
H^2 (X, \mathcal{H}^2 (\Z / m \Z (2))) \simeq CH^2 (X)/m
$$
(see, e.g., the proof of \cite[Theorem 7.7]{BO}). Thus, using Proposition \ref{P-BlochOgusSS}(c), we may rewrite (\ref{E-BOChow1}) as
\begin{equation}\label{E-BOCHow2}
\he^3 (X, \Z / m \Z (2)) \to H^3_{\text{ur}} (F(X), \Z / m \Z(2)) \to CH^2 (X)/m \to \he^4 (X, \Z m \Z (2)).
\end{equation}
In view of Corollary \ref{C-FinEt}, we therefore have

\begin{prop}\label{P-4.3}
Let $K$ be a field, $X$ a smooth, geometrically integral algebraic variety over $K$, and $m$ a positive integer prime to ${\rm char}~K.$ If $K$ is of type $\Fpm$, then the
unramified cohomology group $H^3_{\mathrm{ur}} (K(X), \Z/ m \Z (2))$ is finite if and only if $CH^2 (X)/m$ is finite. In particular, $H^3_{\mathrm{ur}} (K(X), \Z/ m \Z (2))$ is finite if $CH^2(X)$ is finitely generated.
\end{prop}

\vskip2mm

\noindent The two statements of Theorem \ref{T-MainThm} are now contained in Propositions \ref{P-4.1} and \ref{P-4.3}, which concludes the proof.

\vskip5mm

\noindent {\bf Remark 4.5.} The question of the finite generation of $CH^2(X)$ (or, at least, the finiteness of $CH^2(X)/m$) is in general a wide-open problem; we mention a couple of cases where the affirmative answer is known.

\vskip2mm

\noindent (a) \parbox[t]{16cm}{(cf. \cite[\S 4.3]{CT-SB}) Let $K$ be a field of characteristic 0 of type $\Fpm$ and $X$ a smooth geometrically, integral algebraic variety of dimension $d$ over $K$. Write $\bar{X} = X \times_K \bar{K}$ and suppose there exists a dominant rational map
$$
\mathbb{A}_{\bar{K}}^{d-1} \times_{\bar{K}} C \dashrightarrow \bar{X},
$$
where $C/ \bar{K}$ is an integral curve. Then, using Corollary \ref{C-FinEt}, the argument in \cite[Theorem 4.3.7]{CT-SB} shows that $CH^2(X)/m$ is finite. One of the key points is that for any smooth algebraic variety $U$ over $K$, the group ${}_mCH^2(U)$ is finite: indeed, as observed in \cite[Corollaire 2]{CTSS}, the Merkurjev-Suslin theorem implies that ${}_mCH^2(U)$ is a subquotient of $\he^3 (U, \Z / m \Z (2)).$}

\vskip2mm

\noindent (b) \parbox[t]{16cm}{Let $X$ be a noetherian scheme. Recall that $CH_0(X)$ is defined as the cokernel of the natural map
$$
\bigoplus_{x \in X_1} K_1 (\kappa(x)) \to \bigoplus_{x \in X_0} K_0 (\kappa(x))
$$
induced by valuations (or, equivalently, coming from the localization sequence in algebraic $K$-theory). One of the key results of higher-dimensional global class field theory is that if $X$ is a regular scheme of finite type over $\Z$, then $CH_0(X)$ is a finitely generated abelian group. This was initially established in the work of Bloch \cite{BlochCFT}, Kato and Saito \cite{Kato-Saito}, and Colliot-Th\'el\`ene, Soul\'e, and Sansuc \cite{CTSS}; a more recent treatment was given by Kerz and Schmidt \cite{KerzSchmidt} based on ideas of Wiesend. In particular, if $X$ is a smooth irreducible algebraic surface over a finite field, then $CH_0(X)$ is finitely generated.}

\vskip3mm

\noindent To conclude this section, let us observe that the last example points to a connection between finiteness properties of unramified cohomology and Bass's conjecture on the finite generation of algebraic $K$-groups. We refer the reader to \cite{Geiss} for a detailed exposition and precise statements of Bass's conjecture and its motivic analogues.

\section{Unramified cohomology of tori}\label{S-5}

In this section, we will establish Theorem \ref{T-MainThm2} on the finiteness of the unramified cohomology of tori and discuss some applications to $R$-equivalence.

First, however, we would like to put in context the definition of the degree 1 unramified cohomology with coefficients in a torus given in \S \ref{S-1}. Let $K$ be a field equipped with a discrete valuation $v$, with valuation ring $\mathcal{O}_v$ and residue field $\kappa(v).$ Suppose that $m$ is a positive integer prime to ${\rm char}~\kappa(v).$ Recall that by absolute purity and the Gersten conjecture for discrete valuation rings, there is an exact sequence
$$
0 \to \he^i(\mathcal{O}_v, \Z/m \Z(j)) \stackrel{f_v}{\longrightarrow} H^i(K, \Z /m \Z(j)) \stackrel{\partial_{\mathcal{O}_v}^i}{\longrightarrow} H^{i-1}(\kappa(v), \Z /m \Z(j-1)) \to 0
$$
where $f_v$ is the natural map and $\partial_{\mathcal{O}_v}^i$ coincides up to sign with the residue map $\partial_v^i$ introduced in \S\ref{S-1} (see \cite[3.3, 3.6]{CT-SB}). In particular, $\ker \partial_v^i = \ker \partial_{\mathcal{O}_v}^i = \he^i(\mathcal{O}_v, \Z/m \Z(j)).$ It follows that if $X$ is a smooth algebraic variety over a field $F$ with function field $F(X)$, then for the unramified cohomology with respect to the geometric places, we have
\begin{equation}\label{E-DefUnr1}
H^i_{{\rm ur}}(F(X), \Z / m \Z(j)) = \bigcap_{P \in X^{(1)}} \he^i(\mathcal{O}_{X,P}, \Z / m \Z (j)),
\end{equation}
where $\mathcal{O}_{X,P}$ denotes the local ring of $X$ at $P.$
Now let $X$ be as above and $\mathbb{G}$ be a group scheme of multiplicative type over $X$. In this case,
we define the degree 1 unramified cohomology $H^1_{ur}(F(X), G)$ of the generic fiber $G = \mathbb{G}_{F(X)}$ by
\begin{equation}\label{E-DefUnr2}
H^1_{{\rm ur}}(F(X),G) = {\rm Im}~(\he^1(X, \mathbb{G}) \stackrel{f}{\longrightarrow} H^1(F(X),G)),
\end{equation}
where $f$ is the natural map. To make sense of this definition, we observe that by \cite[\S 6]{CTS}, we have
$$
H^1_{{\rm ur}}(F(X),G) = \bigcap_{P \in X^{(1)}} {\rm Im}~f_P^G,
$$
where
$$
f_P^G \colon \he^1(\mathcal{O}_{X,P}, \mathbb{G}) \to H^1(F(X),G)
$$
are the natural maps, which are injective for all points $P$ of codimension 1 according to \cite{CTSCoh}. Comparing (\ref{E-DefUnr2}) with (\ref{E-DefUnr1}), we see that the definition of degree 1 unramified cohomology with coefficients in a group of multiplicative type (in particular, a torus) that we adopted in \S\ref{S-1} is consistent
with the usual definition using residue maps in the case of finite coefficients.

Returning now to the set-up of Theorem \ref{T-MainThm2}, let $X$ be a smooth, geometrically integral variety over a field $K$ and $\mathbb{T}$ be a torus over $X$. Set $T = \mathbb{T}_{K(X)}$ to be the generic fiber of $\mathbb{T}$ and denote by $K(X)_T$ the minimal splitting field of $T$ inside a fixed separable closure $\overline{K(X)}$ of $K(X).$
We assume that $K$ is of type $\Fpm$ for some positive integer $m$ that is prime to ${\rm char}~K$ and divisible by $n = [K(X)_T :K(X)]$.

\vskip2mm

\noindent {\it Proof of Theorem \ref{T-MainThm2}.} First, as in the proof of Corollary  \ref{C-FinTori}(b), Hilbert's Theorem 90 implies that $H^1(K(X), T) = H^1 (K(X)_T/K(X), T)$ is a group of exponent $n.$ Consequently, the natural map $\he^1(X, \mathbb{T}) \to H^1(K(X),T)$ factors through $\he^1(X, \mathbb{T})/n \he^1(X, \tT)$, so, to complete the proof, we need to show that the latter quotient group is finite. For this, we observe that the Kummer sequence
$$
1 \to \tT[n] \to \tT \stackrel{x \mapsto x^n}{\longrightarrow} \tT \to 1
$$
yields an embedding
\begin{equation}\label{E-Embedding}
\he^1(X, \tT)/n \he^1(X, \tT) \hookrightarrow \he^2(X, \tT[n]).
\end{equation}
To see that $\he^2(X, \tT[n])$ is finite, we note that since $X$ is a smooth algebraic variety, $\mathbb{T}$ is isotrivial by \cite[Expos\'e X, Th\'eor\`eme 5.16]{SGA3}, so there exists a connected principal \'etale Galois cover $Y \to X$ with (finite) Galois group $\mathfrak{G}$ that splits $\mathbb{T}.$  Then $\tT[n]_Y$ is a product of groups of the form $\mu_{n', Y}$ with $n' \vert n$, and it follows from Remark 2.5 and Corollary \ref{C-FinEt} that the groups $\he^i(Y, \tT[n]_Y)$ are finite for all $i \geq 0$. Then the groups $H^p(\mathfrak{G}, \he^q(Y, \tT[n]_Y))$ are finite, and the Hochschild-Serre spectral sequence
$$
E_2^{p,q} = H^p(\mathfrak{G}, \he^q(Y, \tT[n]_Y)) \Rightarrow \he^{p+q}(X, \tT[n])
$$
gives the finiteness of $\he^2(X, \tT[n])$, as needed. $\Box$

\vskip2mm

\noindent {\bf Remark 5.1.} Our proof actually shows that if $K$ is of type $\Fpm$ and $m H^1(K(X), T) = 0$, then $H^1_{{\rm ur}}(K(X), T)$ is finite.



\vskip2mm

As we already mentioned in \S \ref{S-1}, our analysis of the finiteness properties of unramified cohomology is motivated in part by the general program of studying absolutely almost simple groups having the same maximal tori over the field of definition. An important component of this work is centered around absolutely almost simple algebraic groups having good reduction at a specified set of discrete valuations. In this regard, we anticipate that Theorem \ref{T-MainThm2} will be useful in addressing the following conjecture (see \cite[\S 6]{CRR5}).

\vskip2mm

\noindent {\bf Conjecture 5.2.} {\it Let $K$ be a field of type $\tF$, and $C$ a smooth affine geometrically integral curve over $K$ with function field $K(C).$ Let $V_0$ be the set of geometric places of $K(C)$, associated with the closed points of $C$. Given an absolutely almost simple simply
connected algebraic $K(C)$-group $G$ for which $\mathrm{char}\: K$ is ``good," the set of $K$-isomorphism classes of $K$-forms of $G$ that have good reduction at all $v \in V_0$ is finite.}

\addtocounter{thm}{2}

\vskip2mm


The proof of Theorem \ref{T-MainThm2} has the following consequence for flasque tori.
First, recall that a torus $\tT$ over a normal, connected scheme $X$ is said to be {\it flasque} if there exists a connected principal \'etale Galois cover $Y \to X$ with (finite) Galois group $\mathfrak{G}$ that splits $\mathbb{T}$ and such that $H^{-1}(\mathfrak{H}, \widehat{\tT}(Y)) = 0$ for all subgroups $\mathfrak{H}$ of $\mathfrak{G}$, where $\widehat{\tT}$ denotes the group of characters of $\tT$ and we consider Tate cohomology (this definition, introduced in \cite{CTS-Flasque} in a slightly more general form, is a natural extension to schemes of the more classical notion of flasque tori over fields, which are discussed in detail in \cite{CTS-Requiv}). An important observation in the theory of flasque tori concerns the existence of {\it flasque resolutions}. Namely, if $X$ is a locally noetherian, normal, connected scheme, then for any torus $\tT$ over $X$, there exists an exact sequence of $X$-tori
$$
1 \to \mathbb{S} \to \mathbb{E} \to \tT \to 1,
$$
where $\mathbb{S}$ is flasque and $\mathbb{E}$ is quasi-trivial (see \cite[Proposition 1.3]{CTS-Flasque}).





\begin{cor}\label{C-FlasqueFinite}
Let $X$ be a smooth, geometrically integral algebraic variety over a field $K$ and $T$ a flasque torus over the function field $K(X)$. Suppose $K$ is of type $\Fpm$
for a positive integer $m$ prime to ${\rm char}~K.$ If $m H^1(K(X), T) = 0$ $($in particular, if $[K(X)_T:K(X)]$ divides $m)$, then $H^1(K(X), T)$ is finite.
\end{cor}
\begin{proof}
Since $T$ is flasque over $K(X)$, it follows from \cite[Proposition 1.5]{CTS-Flasque} that there exists a flasque torus $\tT$ over a Zariski open subset $U \subset X$ whose generic fiber is $T$. According to \cite[Theorem 2.2]{CTS-Flasque}, the fact that $\tT$ is flasque implies that
the natural map $\he^1(U, \tT) \to H^1(K(X), T)$ is surjective.
On the other hand, since $m H^1(K(X), T) = 0$,
this maps factors through the quotient $\he^1(U, \tT)/m \he^1(U, \tT)$. But, as we have seen in the proof of Theorem \ref{T-MainThm2}, the latter quotient is finite, and our claim follows.
\end{proof}

We would like to conclude this section with the following application of Corollary \ref{C-FlasqueFinite} to $R$-equivalence on algebraic tori (see, e.g., \cite[\S 4]{CTS-Requiv} for the relevant definitions).


\begin{cor}\label{C-RequivFinite}
Suppose $X$ is a smooth, geometrically integral algebraic variety over a field $K$ and $T$ is a torus defined over the function field $K(X).$
Assume that $K$ is of type $\Fpm$ for some positive integer $m$ that is prime to ${\rm char}~K$ and divisible by $n = [K(X)_T :K(X)]$. Then
the group of $R$-equivalence classes $T(K(X))/R$ is finite.
\end{cor}
\begin{proof}
Let
$$
1 \to S \to E \to T \to 1
$$
be a flasque resolution of $T$ over $K(X).$ Then, according to \cite[\S 5, Th\'eor\`eme 2]{CTS-Requiv} and \cite[Theorem 3.1]{CTS-Flasque}, there is an isomorphism of abelian groups
$$
T(K(X))/R \stackrel{\sim}{\longrightarrow} H^1(K(X), S).
$$
On the other hand, it follows from \cite[Proposition 6]{CTS-Requiv} and the restriction-corestriction sequence that $n H^1(K(X), S) = 0.$ Since $S$ is flasque, Corollary \ref{C-FlasqueFinite} implies that $H^1(K(X), S)$ is finite, which completes the proof.
\end{proof}


\bibliographystyle{amsplain}

\end{document}